\documentclass[12pt]{amsart}
\usepackage[utf8]{inputenc}
\usepackage{amsthm}
 \usepackage{amssymb,fullpage}
\usepackage{amsmath}
\usepackage{soul,hyperref,color}
\theoremstyle{definition}
\newtheorem{definition}{Definition}[section]
\newtheorem{remark}{Remark}
\newtheorem{theorem}{Theorem}
\newtheorem{corollary}{Corollary}
\newtheorem{example}{Example}

\newtheorem{lemma}{Lemma}
\theoremstyle{empty}
\newtheorem{duplicate}{Theorem}
\title{Generalized Lattice Point Visibility in $\mathbb{N}^k$}
\author{Carolina Benedetti}
\address{Departamento de Matem\'aticas\\Universidad de los Andes\\Bogot\'a\\Colombia} 
\email{c.benedetti@uniandes.edu.co}
\thanks{C. Benedetti was supported by grant FAPA of the Faculty of Science at Universidad de los Andes.}

\author{Santiago Estupi\~nan}
\address{Departamento de Matem\'aticas\\Universidad de los Andes\\Bogot\'a\\Colombia}
\email{s.estupinan10@uniandes.edu.co}
\thanks{}

\author{Pamela E. Harris}
\address{Department of Mathematics and Statistics, Williams College, United States}
\email{peh2@williams.edu}
\thanks{P.\,E. Harris was supported by NSF award DMS-1620202.}

\DeclareMathOperator{\lcm}{lcm}

\date{\today}
\newcommand\bb{{\bf{b}}}
\newcommand\ff{{\bf{f}}}
\newcommand\g{{\bf{g}}}
\newcommand\nn{{\bf{n}}}
\newcommand\Z{{\mathbb{Z}}}
\newcommand\N{{\mathbb{N}}}
\newcommand\R{{\mathbb{R}}}
\newcommand\Q{{\mathbb{Q}}}
\newcommand\F{{\mathcal{F}}}

\begin{document}

\maketitle

\begin{abstract}
    A lattice point $(r,s)\in\mathbb{N}^2$ is said to be visible from the origin if no other integer lattice point lies on the line segment joining 
    the origin and $(r,s)$. 
    It is a well-known result that the proportion of lattice points visible from the origin is given by $\frac{1}{\zeta(2)}$, where $\zeta(s)=\sum_{n=1}^\infty\frac{1}{n^s}$ denotes the Riemann zeta function. 
    Goins, Harris, Kubik and Mbirika, generalized the notion of lattice point visibility by saying that for a fixed $b\in\mathbb{N}$, a lattice point $(r,s)\in\mathbb{N}^2$ is $b$-visible from the origin if no other lattice point lies on the graph of a function $f(x)=mx^b$, for some $m\in\mathbb{Q}$, between the origin and $(r,s)$. In their analysis they establish that for a fixed $b\in\mathbb{N}$, the proportion of $b$-visible lattice points is $\frac{1}{\zeta(b+1)}$, which generalizes the result in the classical lattice point visibility setting. In this short note we give an $n$-dimensional notion of $\bf{b}$-visibility that recovers the one presented by Goins et. al. in $2$-dimensions, and the classical notion in $n$-dimensions. We prove that for a fixed ${\bf{b}}=(b_1,b_2,\ldots,b_n)\in\mathbb{N}^n$ the proportion of ${\bf{b}}$-visible lattice points is given by $\frac{1}{\zeta(\sum_{i=1}^nb_i)}$. 
    
    Moreover, we propose a $\bf{b}$-visibility notion for vectors $\bf{b}\in \Q_{>0}^n$, and we show that by imposing weak conditions on those vectors one obtains that the density of ${\bf{b}}=(\frac{b_1}{a_1},\frac{b_2}{a_2},\ldots,\frac{b_n}{a_n})\in\mathbb{Q}_{>0}^n$-visible points is $\frac{1}{\zeta(\sum_{i=1}^nb_i)}$. Finally, we give a notion of visibility for vectors $\bf{b}\in (\Q^{*})^n$, compatible with the previous notion, that recovers the results of Harris and Omar for $b\in \Q^{*}$ in $2$-dimensions; and show that the proportion of $\bf{b}$-visible points in this case only depends on the negative entries of $\bf{b}$.
\end{abstract}

\section{Introduction}
In classical lattice point visibility, a point $(r,s)$ in the integer lattice $\Z^2$ is said to be visible (from the origin) if the line segment joining the origin $(0,0)$ and the point $(r,s)$ contains no other integer lattice points. One well-known result
is that the proportion of visible lattice points in $\Z^2$ is given by $1/\zeta(2)$, where $\zeta(s)=\sum_{n=1}^{\infty}\frac{1}{n^s}$  denotes the Riemann zeta function \cite{apostol}. In fact, a similar argument establishes that for $k\geq 2$ the proportion of visible lattice points in $\Z^k$ (analogously defined)
is given by $1/\zeta(k)$ \cite{Christopher}.

In 2017, Goins, Harris, Kubik and Mbirika generalized the classical definition of lattice point visibility by fixing $b\in\mathbb{N}$ and considering curves of the form $f(x)=ax^b$ with $a\in \mathbb{Q}$, see~\cite{GHKM2017}. 
In this new setting, a lattice point $(r,s)\in\mathbb{N}^2$ is said to be $b$-visible if it lies on the graph of a curve of the form $f(x)=ax^b$ with $a\in \mathbb{Q}$ and there are no other integer lattice points lying on this curve between $(0,0)$ and $(r,s)$. Hence, setting $b=1$ recovers the classical definition of lattice point visibility.
In the $b$-visibility setting, Goins et. al.  established that the proportion of $b$-visible lattice points in $\N^2$ is given by $\frac{1}{\zeta(b+1)}$. Harris and Omar expanded this work to power functions with rational exponents by establishing that the proportion of $(b/a)$-visible lattice points in $\N_a^2$ (the set nonnegative integers that are $a$th powers) is given by $\frac{1}{\zeta(b+1)}$, and the proportion of $(-b/a)$-visible lattice points in $\N_a^2$ is given by $\frac{1}{\zeta(b)}$  \cite{HarrisOmar2017}.

In this work we extend the notion of visibility to lattice points in $\N^k$. 

First we do it for the classical notion and the one developed my Goins et al, by fixing  $\bb=(b_1,b_2,\ldots,b_k)\in\N^k$, and saying that a lattice point $\nn=(n_1,n_2,\ldots,n_k)\in\N^k$ is $\bb$-visible if there does not exists a positive real number $0<t<1$ such that $(n_1t^{b_1},n_2t^{b_2},\ldots,n_kt^{b_k})$ is a lattice point in $\mathbb{N}^k$ (Definition \ref{def:simpleone}). Using this definition, we prove the following result regarding the proportion of $\bb$-visible points in $\N^k$. 

\begin{theorem}
\label{thm:main}
Fix $k\in\mathbb{N}$ and ${\bf{b}}=(b_1,b_2,\ldots,b_k)\in \mathbb{N}^k$, with ${\bf{b}}$ satisfying the condition that $\gcd(b_1,b_2,\ldots,b_k)=1$.
Then the proportion of points ${\bf{n}}\in \mathbb{N}^k$ that are $\bf{b}$-visible is $\frac{1}{\zeta\left(\sum_{i=1}^{k}b_i\right)}$.
\end{theorem}

For an $k\in\N$, note that setting
$\bb=(1,1,\ldots,1)\in\mathbb{N}^k$ in Theorem \ref{thm:main} recovers the proportionality result in classical lattice point visibility.

To state our next result, we consider $a_1,a_2,\ldots,a_k\in\mathbb{N}$ and define $\alpha:=\lcm(a_1,\ldots,a_k)$. Also for each $1\leq i\leq k$ we let $\mathbb{N}_{\frac{\alpha}{a_i}}$ denote the set of integers of the form $\ell^{\frac{\alpha}{a_i}}$ with $\ell\in\N$, and we say that $\gcd(\bb)=1$ if there is an integer linear combination of the entries of $\bb$ equal to $1$. 

\begin{theorem}\label{thm:main2}
Fix $k\in\mathbb{N}$ and  ${\bf{b}}=(\frac{b_1}{a_1},\frac{b_2}{a_2},\ldots,\frac{b_k}{a_k})\in{\mathbb{Q}_{>0}}^k$, with $\bf{b}$ satisfying the conditions that $a_1,a_2,\ldots,a_k\in\mathbb{N}$ and $\gcd(\frac{b_1}{a_1},\frac{b_2}{a_2},\ldots,\frac{b_k}{a_k})=1$.Then the proportion of points in $\mathbb{N}_{\frac{\alpha}{a_1}}\times \cdots \times \mathbb{N}_{\frac{\alpha}{a_n}}$ that are $\bf{b}$-visible is $\frac{1}{\zeta(\sum_{i=1}^kb_i)}$.
\end{theorem}

\begin{theorem}\label{thm:main3}
Let ${\bf{b}}=(\frac{b_1}{a_1},\ldots,\frac{b_k}{a_k})\in({\mathbb{Q}^{*}})^k$ be such that its negative entries are indexed by the set $J \subseteq [k]$, with $\bf{b}$ satisfying the conditions that $a_1,a_2,\ldots,a_k\in\mathbb{N}$ and  $\gcd(\frac{b_1}{a_1},\frac{b_2}{a_2},\ldots,\frac{b_k}{a_k})=1$. Then the proportion of points in $\mathbb{N}_{\frac{\alpha}{a_1}}\times \cdots \times \mathbb{N}_{\frac{\alpha}{a_n}}$ that are $\bf{b}$-visible is $\frac{1}{\zeta(\sum_{j\in J}|b_j|)}$.
\end{theorem}

Note that  
setting $k=2$ and $\bb=(1,b)\in\mathbb{Q}^2$ in Theorems \ref{thm:main2} and \ref{thm:main3} recovers the analogous results in the $b$-visibility setting as considered by Harris and Omar \cite{HarrisOmar2017}, thereby extending them to $n$-dimensions.

We organize this work as follows: Section \ref{sec:background} is separated into two parts. The first considers $\bb\in\N^k$ and provides the definition of $\bb$-visibility along with a number-theoretic characterization of what it means for a lattice point to be $\bb$-visible. The second part extends this definition and provides an analogous number-theoretic result in the case where $\bb\in\Q^k$. Section \ref{sec:mainresult} contains the results on the proportion of lattice points that are $\bb$-visible, work which  establishes our main results (Theorem \ref{thm:main}-\ref{thm:main3}).

\section{Background}\label{sec:background}

Goins et. al. gave the following definition in \cite{GHKM2017}: Fix $b\in\mathbb{N}$, then a lattice point $(r,s)\in\mathbb{N}^2$ is said to be $b$-visible if it lies on the graph of a curve of the form $f(x)=ax^b$ with $a\in \mathbb{Q}$ and there does not exist $(r',s')\in\mathbb{N}^2$ on $f$ with $0<r'<r$. 
This definition of $b$-visibility is dependent on a power function $f(x)$ on which the lattice point $(r,s)$ lies.  However, we start by presenting an equivalent definition which is a parametrized version of the $b$-visibility  and bypasses the need for such a function  $f$. 

\begin{definition}\label{def:par_vis}
Fix $b\in\mathbb{N}$. A lattice point $(r,s)\in\N^2$ is said to be \emph{$b$-visible} if there does not exist a real number $0<t<1$ such that
$(rt,st^b)\in\mathbb{N}^2$.
\end{definition}

Definition \ref{def:par_vis} will be the key in defining $\bb$-visibility to $\N^k$, for $\bb\in\Q^k$. In order to stay consistent with known literature on lattice point visibility, we begin next section by stating a general definition of $\bb$-visibility which depends on lattice points lying on the graph of certain real-valued functions. We then show that this definition of $\bb$-visibility is in fact independent of what function the point lies on. This leads naturally to a  definition of lattice point $\bb$-visibility relying solely on a number-theoretic description of the lattice point $\nn\in\N^k$. We separate the remainder of this section into three cases: $\bb\in\N^k$, $\bb\in\Q_{>0}^k$, and the case where $\bb\in\Q^k$. 

\subsection{On $\bb$-visible lattice points, with $\bb\in\N^k$}
In this section we consider $\bb\in\N^k$ and begin by presenting a definition of $\bb$-visibility, which depends on a lattice point lying on a certain real-valued function.

\begin{definition}\label{def1}
Fix $\bb=(b_1,b_2,\ldots,b_k)\in\mathbb{N}^k$ and define 
\[\F(\bb):=\{\ff:\mathbb{R}\to\mathbb{R}^k\; |\; \ff(t)=(m_1t^{b_1},m_2t^{b_2},\ldots,m_kt^{b_k})\mbox{ where }m_1,m_2,\ldots,m_k\in\N\}.\]
If $\nn=(n_1,n_2,\ldots,n_k)\in\mathbb{N}^k$ and there exists $\ff\in\mathcal{F}(\bb)$ such that
\begin{enumerate}
    \item $\ff(t)=\nn$ for some $t\in \R_{>0}$ and
    \item there does not exist $0<t'<t$ such that $\ff(t')\in\N^k$
\end{enumerate}
then $\nn$ is said to be \emph{$\bb$-visible with respect to $\ff$}. If Condition (1) is satisfied, but (2) is not, then we say that \emph{$\nn$ is  $\bb$-invisible with respect to $\ff$}.
\end{definition}

In order to illustrate Definition \ref{def1} we present the following example.

\begin{example}
The point $(4,16,40,128)$ lies on the graph of $\ff(t)=(4t^{2},16 t^{4},40 t^{3},128 t^{7})$ since $\ff(1)=(4,16,40,128)$, but it is not $(2,4,3,7)$-visible, because $$\ff\left(\frac{1}{2}\right)=\left(4 \left(\frac{1}{2}\right)^2,16 \left(\frac{1}{2}\right)^4,40 \left(\frac{1}{2}\right)^3,128 \left(\frac{1}{2}\right)^7\right)=(1,1,5,1)\in\N^4.$$ However, $(1,1,5,1)$ is $(2,4,3,7)$-visible since it lies on the curve $\g(t)=(t^{2},t^{4},5 t^{3}, t^{7})$ as $\g(1)=(1,1,5,1)$ and there does not exist $0<t'<1$ such that $\g(t')\in\N^4$.
\end{example}

We remark that the classical notion of visibility in $\N^k$ is the particular case of taking $\bb = (1,\dots,1)\in\N^k$ in Definition \ref{def1}, while $b$-visibility corresponds to taking $\bb = (1,b)$.
For a fixed $\bb\in\N^k$ we now show that the $\bb$-visibility of the lattice point $\nn\in\N^k$ is independent of the choice of function $\ff$ satisfying Definition \ref{def1}. This is the content of the following result.

\begin{lemma}\label{indep_visibility}
Fix $\bb=(b_1,\ldots,b_k),\;\nn=(n_1,\ldots,n_k)\in\N^k$ and suppose that $\nn$ is $\bb$-visible with respect to $\ff\in\F(\bb)$.
If $\g\in\F(\bb)$ and $\g(r)=\nn$ for some $r\in\R_{>0}$, then $\nn$ is $\bb$-visible with respect to $\g$. 
\end{lemma}
\begin{proof}
By hypothesis $\ff(s)=(m_1s^{b_1},\ldots,m_ks^{b_k})=\nn$ for some $s\in\R_{>0}$ and $m_i\in\N$ for all $1\leq i\leq k$. Likewise, $\g(r)=(u_1r^{b_1},\dots,u_kr^{b_k})=\nn$ for some $r\in\R_{>0}$ and $u_i\in\N$ for all $1\leq i\leq k$. Thus
$(m_1s^{b_1},\dots,m_ks^{b_k})=(u_1r^{b_1},\dots,u_kr^{b_k})$, which implies
$m_i(s/r)^{b_i}=u_i$ for each $1\leq i\leq k$. If there exists $0<t<r$ such that $\g(t)=(v_1,\dots,v_k)\in\mathbb N^k$ then for each $1\leq i\leq k$ we have
$$v_i=u_it^{b_i}=m_i(s/r)^{b_i}t^{b_i}=m_i(st/r)^{b_i}.$$
Since $t<r$ then $st/r<s$, and thus $\mathbf f(st/r)\in\mathbb N^k$, which contradicts the fact that $\mathbf n$ is $\bb$-visible with respect to $\mathbf f$. Hence, such a $t$ can not exists and we conclude that $\mathbf n$ is $\bb$-visible with respect to $\mathbf g$.
\end{proof}

Note that for any $\bb=(b_1,b_2,\ldots,b_k)$ and $\nn=(n_1,n_2,\ldots,n_k)$ in $\N^k$, $\nn$ lies on the curve $\ff(t)=(n_1t^{b_1}, n_2t^{b_2},\ldots,n_kt^{b_k})\in\F(\bb)$ since $\ff(1)=\nn$. Thus, by Lemma \ref{indep_visibility} the $\bb$-visibility of $\nn$ can be reduced to determining the existence of $0<t<1$ for which $\ff(t)\in\mathbb N^k$.
Hence, from now on we refer to the $\bb$-visibility of a lattice point $\nn$ with no reference to a specific function and we use the following definition of $\bb$-visibility, whenever  $\bb\in\N^k$.

\begin{definition}\label{def:simpleone}
Fix $\bb=(b_1,b_2,\ldots,b_k)\in\N^k$. Then the lattice point $\nn=(n_1,n_2,\ldots,n_k)\in\N^k$ is said to be $\bb$-visible if there does not exist a positive real number $0<t<1$ such that $(n_1t^{b_1},n_2t^{b_2},\ldots,n_kt^{b_k})\in\N^k$.
\end{definition}

We would like to have a number theoretic characterization of the points that are $\bb$-visible according to the above definition; and one would expect it to be almost identical to the characterization given by Goins et. al. However, a new obstacle arises in this context:

For the former notions of visibility, if a point ${\bf{n}}=(n_1,n_2)\in \N^2$ was not $b$-visible then there had to exist a rational number $t<1$ such that $(tn_1,t^bn_2)\in \N^2$ and $t^bn_1=\frac{n_1}{n_2^b}(tn_2)^b$. This does not need to be the case now. Take for instance the point $(2,4)\in \N^2$. The only point with integer coordinates $(2t^2,4t^4)$ for $0<t<1$ is $(1,1)$ which is obtained by making $t=\frac{1}{\sqrt{2}}$. Hence, in accordance with our definition, $(2,4)$ is not $(2,4)$-visible, but this can only be noticed if $t$ is allowed to be an irrational number. 

This can be avoided if $\bf{b}$ satisfies $\gcd(\bb)=1$ where $\gcd(\bb)$ is the greatest common divisor of the entries of $\bb$. Indeed, in that case there must exist integers satisfying $m_1b_1+m_2b_2+\cdots +m_kb_k=1$. So, by hypothesis one has that $(n_1t^{b_1},n_2t^{b_2},\ldots,n_kt^{b_k})\in\N^k$, and then:

$$(n_1t^{b_1})^{m_1}\ldots (n_kt^{b_k})^{m_k}=(n_1^{m_1})\ldots (n_k^{m_k})t^{m_1b_1+m_2b_2+\cdots +m_kb_k}=(n_1^{m_1})\ldots (n_k^{m_k})t\in \Q.$$

Thus $t\in \Q$, and we only need to worry about the case $\gcd(\bb)>1.$

\begin{lemma}
A point $\bf{n}\in \N^k$ is $\bb$-visible if and only if it is $\left( \frac{1}{\gcd(\bb)}\bb\right)$-visible.
\end{lemma}
\begin{proof}
Suppose that $\bf{n}$ is not $\bb$-visible. Then there exists $0<t<1$ with $(n_1t^{b_1},n_2t^{b_2},\ldots,n_kt^{b_k})\in\N^k$. As a result $t^{\gcd(\bb)}\in\Q$ by an argument analogous to the one above. But then:

$$(n_1(t^{\gcd(\bb)})^{\frac{b_1}{\gcd(\bb)}},n_2(t^{\gcd(\bb)})^{\frac{b_2}{\gcd(\bb)}},\ldots,n_k(t^{\gcd(\bb)})^{\frac{b_k}{\gcd(\bb)}})\in\N^k$$

Where $t^{\gcd(\bb)}<1$ because $t<1.$ So $\bf{n}$ is not $\left( \frac{1}{\gcd(\bb)}\bb\right)$-visible.

For the converse direction, if $(n_1(t)^{\frac{b_1}{\gcd(\bb)}},n_2(t)^{\frac{b_2}{\gcd(\bb)}},\ldots,n_k(t)^{\frac{b_k}{\gcd(\bb)}})\in\N^k$, for $0<t<1$, then taking $t'=t^{\frac{1}{\gcd(\bb)}}$ evinces that $\bf{n}$ is not $\bb$ visible either. 
\end{proof}

In light of the foregoing, we will only study $\bb$-visibility for $\gcd(\bb)=1$.

\begin{theorem}\label{thm:ifandonlyif1}
Fix $\bb=(b_1,\ldots,b_k)\in \N^k$ satisfying $\gcd(\bb)=1$. Then the lattice point $\nn=(n_1,\ldots,n_k)\in\N^k$ is $\bb$-visible if and only if there does not exist a prime $p$, such that $p^{b_i}|n_i$ for all $i=1,2,\ldots, k$.
\end{theorem}
\begin{proof}
Suppose that ${\bf{n}}$ is ${\bf{b}}$-visible. If there exists a prime $p$ satisfying $p^{b_i}|n_i$ for all $i=1,\ldots,k$, then the point $\nn'=({n_1}(\frac{1}{p})^{b_1}, {n_2}(\frac{1}{p})^{b_2},\ldots,n_k(\frac{1}{p})^{b_k} )$ is an integer lattice point. Since $\frac{1}{p}<1$, it follows that $\nn$ is not $\bb$-visible, which is a contradiction.

On the other hand let us suppose that ${\bf{n}}$ is not ${\bf{b}}$-visible. Then there exists  $0<t<1$ such that $({n_1}t^{b_1}, {n_2} t^{b_2},\ldots,n_k t^{b_k})\in \mathbb{N}^k$. Since $\gcd(\bb)=1$ and $t\in \Q$, we can take $a$ and $c$ with $\gcd(a,c)=1$, and let $t=\frac{a}{c}$. Then $c^{b_i}|n_i$ for all $i=1,\ldots,k$. Thus any prime factor $p$ of $c$ yields the desired result.
\end{proof}

\begin{definition}\label{def:brelativelyprime1}
Let $\bb,\nn\in \N^k$. We  say that $\nn$ is \emph{$\bb$-relatively prime} if there does not exist a prime $p$ such that $p^{b_i}|n_i$ for all $i=1,\ldots,k$.
\end{definition}

In light of Definition \ref{def:brelativelyprime1}, we can restate Theorem \ref{thm:ifandonlyif1} as follows.

\begin{corollary}\label{cor:numbertheoretic}
Fix $\bb\in \N^k$ satisfying $\gcd(\bb)=1$. Then the lattice point $\bb\in \N^k$ is $\bf{b}$-visible if and only if $\nn$ is $\bb$-relatively prime.
\end{corollary}

Note that Theorem \ref{thm:ifandonlyif1} and Corollary \ref{cor:numbertheoretic} present number-theoretic characterizations of $\bb$-visible lattice points whenever $\bb\in\N^k$. 

\subsection{On $\bb$-visible lattice points, with $\bb\in\Q^k$} 
We now extend the definition of $\bb$-visibility to  the case where $\bb\in\Q^k$, and present number-theoretic results analogous to Theorem \ref{thm:ifandonlyif1} and Corollary \ref{cor:numbertheoretic}.
These results generalize the $b$-visibility proportionality results for $b\in\mathbb{Q}$ obtained by Harris and Omar in \cite{HarrisOmar2017} to lattice points in $\N^k$. 

In what follows, we begin by considering $\bb\in\mathbb{Q}_{>0}^k$, i.e. $\bb$ with all positive rational entries, and later we consider the case where $\bb$ has some negative entries.

\begin{definition}
\label{def:bvisibilityinndimensions2}
Fix $\bb=(\frac{b_1}{a_1},\frac{b_2}{a_2},\ldots,\frac{b_k}{a_k})\in\mathbb{Q}_{>0}^k$ and suppose that $\nn=(n_1,n_2,\ldots,n_k)\in\N^k$ lies on the curve \[\ff(t)=(m_{1}t^{\frac{b_1}{a_1}},m_{2}t^{\frac{b_2}{a_2}},\cdots,m_{k}t^{\frac{b_k}{a_k}})\] for some ${\bf{m}}=(m_1,m_2,\ldots,m_k) \in\N^k$. Then the point $\nn$ is said to be \emph{$\bb$-visible (with respect to $\ff$)} if there does not exist another point in $\N^k$ on the graph of $\ff(t)$ lying between the origin and~$\nn$.
If $\nn$ is not $\bb$-visible, then we say $\nn$ is \emph{$\bb$-invisible (with respect to $\ff$)}.
\end{definition}

As before, we note that for any $\bb=(\frac{b_1}{a_1},\frac{b_2}{a_2},\ldots,\frac{b_k}{a_k})\in\mathbb{Q}_{>0}^k$ the lattice point  \[\nn=(n_1,n_2,\ldots,n_k)\in\N^k\]
lies on the curve $\ff(t)=(n_1t^{\frac{b_1}{a_1}}, n_2t^{\frac{b_2}{a_2}},\ldots,n_kt^{\frac{b_k}{a_k}})$ since $\ff(1)=\nn$. 
Thus, for $\bb\in\Q_{>0}^k$, the $\bb$-visibility of $\nn$ can again be reduced to determining the existence of a real number  $0<t<1$ for which $\ff(t)\in\mathbb N^k$. 
Moreover, note that the statement and proof of Lemma \ref{indep_visibility} hold when $\bb\in\Q_{>0}^k$. 
Thus, the  $\bb$-visibility of $\nn\in\N^k$ is independent of the function~$\ff$  when $\bb\in\Q_{>0}^k$.

For $a_1,a_2,\ldots,a_k\in\mathbb{N}$ we let $\alpha:=\lcm(a_1,\ldots,a_k)$ and for each $1\leq i\leq k$ we let $\mathbb{N}_{\frac{\alpha}{a_i}}$ denote the set of integers of the form $\ell^{\frac{\alpha}{a_i}}$ with $\ell\in\N$. 
We now state and prove the following technical result.

\begin{lemma}\label{lem:techresult}
Let $\bb=(\frac{b_1}{a_1},\ldots,\frac{b_k}{a_k})\in \mathbb{Q}_{>0}^k$. 
Then the $\bb$-visibility of $\nn\in\N^k$ with respect to the family of functions $\ff(t)=(m_1t^{\frac{b_1}{a_1}},\ldots,m_kt^{\frac{b_k}{a_k}})$ with $(m_1,\ldots,m_k)\in \N^k$ is equivalent to the $\bb$-visibility of $\nn\in\N^k$ with respect to the family of functions $\g(t)=(m_1t^{\frac{b_1}{a_1}},\ldots,m_kt^{\frac{b_k}{a_k}})$ with $(m_1,\ldots,m_k)\in \mathbb{N}_{\frac{\alpha}{a_1}}\times \ldots \times \mathbb{N}_{\frac{\alpha}{a_k}}.$
\end{lemma}
\begin{proof}
Let $\nn\in \mathbb{N}_{\frac{\alpha}{a_1}}\times \ldots \times \mathbb{N}_{\frac{\alpha}{a_k}}$, and suppose that $\nn$ is not $\bb$-visible with respect to $\ff(t)=(m_1t^{\frac{b_1}{a_1}},\ldots,m_kt^{\frac{b_k}{a_k}})$ with $(m_1,\ldots,m_k)\in \N^k$. Then there exist $t',t''\in \mathbb{R}$ with $t''<t'$, and $\ff(t'')\in \mathbb{N}_{\frac{\alpha}{a_1}}\times \ldots \times \mathbb{N}_{\frac{\alpha}{a_k}}$, $\ff(t')=\nn.$ But then 

$$\left(m_1(t')^{\frac{b_1}{a_1}}\left(\frac{t''}{t'}\right)^{\frac{b_1}{a_1}},\ldots,m_k(t')^{\frac{b_k}{a_k}}\left(\frac{t''}{t'}\right)^{\frac{b_k}{a_k}}\right)=\left(n_1\left(\frac{t''}{t'}\right)^{\frac{b_1}{a_1}},\ldots,n_k\left(\frac{t''}{t'}\right)^{\frac{b_k}{a_k}}\right).$$

And recall that $\nn\in \mathbb{N}_{\frac{\alpha}{a_1}}\times \ldots \times \mathbb{N}_{\frac{\alpha}{a_k}}$, so the right side of the last equation is equal to $\g(\frac{t''}{t'})$ for a function $\g(t)$ belonging to the second family of functions, and because $t''<t'$ we have $\frac{t''}{t'}<1$, so it is also not visible with respect to $\g(t)$. 

The other direction is clear because the second family of functions $\g(t)$ is included in the family of functions $\ff(t)$.
\end{proof}

Note that Lemma \ref{lem:techresult} allows us to restrict our $\bb$-visibility study to lattice points~ ${\bf{n}}\in \mathbb{N}_{\frac{\alpha}{a_1}}\times \cdots \times \mathbb{N}_{\frac{\alpha}{a_k}}$.
In order to present an analogous result to Theorem \ref{thm:ifandonlyif1} we need the following technical result.

\begin{lemma}\label{lemma:1}
Let $b,c\in\N$ with $gcd(b,c)=1$. If $t\in \mathbb{Q}$ and $t^{\frac{b}{c}} \in \mathbb{Q}$, then $t^{\frac{1}{c}}\in \mathbb{Q}$.
\end{lemma}
\begin{proof}
Let $t=\frac{r}{s}$, we can take $r$ and $s$ so that $\gcd(r,s)=1$. Suppose by contradiction that $t^{\frac{1}{c}}\notin \mathbb{Q}$. Then, necessarily $r$ or $s$ is not a perfect $c$-th power. Without loss of generality, assume $r$ is not a perfect $c$-th power. Then $r^b$ is also not a perfect $c$-th power since $\gcd(b,c)=1$ by hypothesis. Hence $r^{\frac{b}{c}}\notin \mathbb{Q}$. There are two cases for $s$, either it is a perfect $c$-th power, or it is not. 

In the first case $s^{\frac{1}{c}}\in \mathbb{Q}$, so $s^{\frac{b}{c}}\in \mathbb{Q}$, thus $\frac{r^{\frac{b}{c}}}{s^{\frac{b}{c}}}=t^{\frac{b}{c}}\notin \mathbb{Q}$ contradicting the hypothesis.
In the second case, both $r^b$ and $s^b$ are not perfect $c$-th powers (again because $\gcd(b,c)=1$). Let $\alpha:=r^b$, $\beta:=s^b$, it follows that $\gcd(\alpha,\beta)=1$. If
$\frac{\alpha^{\frac{1}{c}}}{\beta^{\frac{1}{c}}}\in \mathbb{Q}$, then $({\frac{\alpha}{\beta}})^{\frac{1}{c}}\in \mathbb{Q}$, hence $(\alpha\cdot {\beta^{-1}})^{\frac{1}{c}}\in \mathbb{Q}$, but we chose $\alpha$ and $\beta$ with $\gcd(\alpha,\beta)=1$, thus {$\alpha\cdot \beta^{-1}$} can not be a perfect $c$-th power, a contradiction.
\end{proof}

\begin{remark}\label{remark:5}
The proof of Lemma \ref{lemma:1} also shows that if $({\frac{\alpha}{\beta}})^{\frac{1}{c}}\in \mathbb{Q}$, for some $\alpha,\beta\in \mathbb{N}$ with~$\gcd(\alpha,\beta)=1$, then both $\alpha$ and $\beta$ must be perfect $c$-th powers. \end{remark}
We are now ready to state our next result. Here we use $\gcd(\bb)=1$ to mean that there is an integer linear combination of the entries of $\bb$ that equals $1$. 

\begin{theorem}\label{thm:ifandonlyif3}
Fix ${\bf{b}}=(\frac{b_1}{a_1},\frac{b_2}{a_2},\ldots,\frac{b_k}{a_k})\in{\mathbb{Q}_{>0}}^k$ with $\gcd(\bb)=1$. Then the lattice
point ${\bf{n}}=({\ell_1}^{\frac{\alpha}{a_1}},{\ell_2}^{\frac{\alpha}{a_2}},\ldots,{\ell_k}^{\frac{\alpha}{a_k}})\in\mathbb{N}_{\frac{\alpha}{a_1}}\times \ldots \times \mathbb{N}_{\frac{\alpha}{a_k}}$ is ${\bf{b}}$-visible if and only if $(\ell_1,\ell_2,\ldots,\ell_k)$ is $(b_1,b_2,\ldots,b_k)$-visible.
\end{theorem}
\begin{proof}
Suppose that ${\bf{n}}$ is ${\bf{b}}$-visible and recall $\alpha:=\lcm(a_1,\ldots,a_k)$. 
 If $(\ell_1,\ell_2,\ldots,\ell_k)$ is not $(b_1,b_2,\ldots,b_k)$-visible,
then there exists a prime $p$ satisfying $p^{b_i}|\ell_i$ for all $i=1,\ldots,k$ and $({\ell_1}\cdot(\frac{1}{p})^{b_1},\ldots,\ell_k\cdot(\frac{1}{p})^{b_k})$ is an integer lattice point. Hence, the point $$\left({{\ell_1}^{\frac{\alpha}{a_1}}}\cdot\left(\frac{1}{p^{\alpha}}\right)^{\frac{b_1}{a_1}}, \ldots,{{\ell_k}^{\frac{\alpha}{a_k}}}\cdot\left(\frac{1}{p^{\alpha}}\right)^{\frac{b_k}{a_k}}\right)=
\left({{\left(\frac{\ell_1}{p^{b_1}}\right)}^{\frac{\alpha}{a_1}}}, \ldots,{{\left(\frac{\ell_k}{p^{b_k}}\right)}^{\frac{\alpha}{a_k}}}\right)$$ 
belongs to the set $\N_{\frac{\alpha}{a_1}}\times\cdots\times \N_{\frac{\alpha}{a_k}}$. 
Recall that
being $\bb$-visible does not depend on the choice of $(m_1,...,m_k)$. This together with
the fact that $\frac{1}{p^\alpha}<1$, implies that ${\bf{n}}$ is ${\bf{b}}$-invisible, which is a contradiction.

Now suppose that ${\bf{n}}$ is not ${\bf{b}}$-visible. Then there exists  $t<1$ such that $({\ell_1}^{\frac{\alpha}{a_1}}\cdot t^{\frac{b_1}{a_1}},\ldots,{\ell_k}^{\frac{\alpha}{a_k}}\cdot t^{\frac{b_k}{a_k}})\in \mathbb{N}_{\frac{\alpha}{a_1}}\times \ldots \times \mathbb{N}_{\frac{\alpha}{a_n}}$. 
In particular, we have that $t^{\frac{b_i}{a_i}}\in \mathbb{Q}$ for all $i=1,\ldots,k$, and since $\gcd({\bf{b}})=1$ it follows that $t\in \mathbb{Q}$.
By Lemma \ref{lemma:1}, we have that $t^{\frac{1}{a_i}}\in \mathbb{Q}$ for $i=1,\ldots,k$, and by Remark \ref{remark:5}, $t=\frac{r}{s}$ for some $r,s\in \mathbb{N}$, where $r$ and $s$ are perfect $a_i$-th powers for all $i=1,\ldots,k$. 

Hence, $s=\prod\limits_{j=1}^{d}p_j^{c_j}$ for some prime numbers $p_1,\ldots, p_\ell$, and integers $c_1,\ldots, c_\ell$ satisfying $\alpha|c_j$ for all $1\leq j\leq d$. Thus, for each $1\leq j\leq d$, there exist $c_j'$  such that $c_j=\alpha\cdot c_j'$. Thus, for all $i=1,\ldots,k$, we have $$\left({\frac{1}{s}}\right)^{\frac{b_i}{a_i}}=\left(\frac{1}{\prod_{j=1}^{d}p_j^{b_i\cdot \alpha \cdot c_j'}}\right)^{\frac{1}{a_i}}=\left(\frac{1}{\prod_{j=1}^{d}p_j^{b_i \cdot c_j'}}\right)^{\frac{\alpha}{a_i}}
$$
Thus,
$\ell_i^{\frac{\alpha}{a_i}}\cdot t^{\frac{b_i}{a_i}} \in \mathbb{N}$ implies that $$\ell_i^{\frac{\alpha}{a_i}}\cdot \left({\frac{1}{s}}\right)^{\frac{b_i}{a_i}}=\left({\frac{\ell_i}{\prod_{j=1}^{d}p_j^{b_i \cdot c_j'}}}\right)^{\frac{\alpha}{a_i}}\in \mathbb{N}.$$

In particular, for any $j=1,\ldots,d$, we have that $p_j^{b_i}|\ell_i$ for all $i=1,\ldots, k$.  Therefore, $(\ell_1,\ell_2,\ldots,\ell_k)$ is not $(b_1,b_2,\ldots,b_k)$-visible, as desired.
\end{proof}

Finally, we give a definition for $\bb$-visibility allowing for negative rational exponents. Let ${\bf{b}}\in ({\mathbb{Q}^{*}})^k$, that is, a $k$-tuple whose entries might be positive or negative rationals, but not $0$. 
This will allow us to generalize the work of Harris and Omar \cite{HarrisOmar2017}. 
From now on, whenever we write a rational as $\frac{b}{a}$, $a$ is assumed to be positive.

\begin{definition}[${\bf{b}}$-visibility for $\bb\in({\mathbb{Q}^{*}})^k$]\label{def:bvisibilityinndimensions3}
Let ${\bf{b}}=(\frac{b_1}{a_1},\ldots,\frac{b_k}{a_k})\in ({\mathbb{Q}^{*}})^k$, and let $\alpha:=\lcm(a_1,\ldots,a_k)$. 
Then an integer lattice point ${\bf{n}}=(n_{1},\ldots,n_{k})\in \mathbb{N}_{\frac{\alpha}{a_1}}\times \ldots \times \mathbb{N}_{\frac{\alpha}{a_k}}$ is \emph{${\bf{b}}$-visible} if the following conditions hold:
\begin{enumerate}
    \item ${\bf{n}}$ lies on the graph of ${\bf{f}}(t)=(m_{1}t^{\frac{b_1}{a_1}},\ldots,m_{k}t^{\frac{b_k}{a_k}})$ for some ${\bf{m}}=(m_1,m_2,\ldots,m_k) \in \mathbb{N}_{\frac{\alpha}{a_1}}\times \cdots \times \mathbb{N}_{\frac{\alpha}{a_k}}$.
    \item There does not exist $t' \in \mathbb{R}$ such that $0<t'<t$, and $(m_{1}{(t')}^{\frac{b_1}{a_1}},\ldots,m_{k}{(t')}^{\frac{b_k}{a_k}})$ is a point in $\mathbb{N}_{\frac{\alpha}{a_1}}\times \ldots \times \mathbb{N}_{\frac{\alpha}{a_k}}$.
\end{enumerate}
If condition (1) is satisfied, but condition (2) is not, then we say that the point ${\bf{n}}$ is  \emph{${\bf{b}}$-invisible}.
\end{definition}

\begin{remark}\label{remark:6}
Being $\bb$-visible does not depend on the choice of $(m_1,...,m_k)$, just as in Definition \ref{def1}.
\end{remark}

\begin{theorem}\label{thm:ifandonlyif4}
Let ${\bf{b}}=(\frac{b_1}{a_1},\ldots,\frac{b_k}{a_k})\in({\mathbb{Q}^{*}})^k$ be such that its negative entries are indexed by the set $J \subseteq [k]$, with $\bf{b}$ satisfying $\gcd(\bb)=1$, then the lattice
point ${\bf{n}}=({\ell_1}^{\frac{\alpha}{a_1}},{\ell_2}^{\frac{\alpha}{a_2}},\ldots,{\ell_k}^{\frac{\alpha}{a_k}})\in\mathbb{N}_{\frac{\alpha}{a_1}}\times \ldots \times \mathbb{N}_{\frac{\alpha}{a_k}}$ is ${\bf{b}}$-visible if and only if there does not exist a prime $p$ such that $p^{|b_j|}|\ell_j$ for all $j\in J$.
\end{theorem}
\begin{proof}
Suppose that there exists a prime $p$ with $p^{|b_j|}|\ell_j$ for all $j\in J$. It follows that the $h$ entry of the point

\begin{align}
\left({{\ell_1}^{\frac{\alpha}{a_1}}}\cdot\left(p^{\alpha}\right)^{\frac{b_1}{a_1}}, {{\ell_2}^{\frac{\alpha}{a_2}}}\cdot\left({p^{\alpha}}\right)^{\frac{b_2}{a_2}},\ldots,{{\ell_k}^{\frac{\alpha}{a_k}}}\cdot\left(p^{\alpha}\right)^{\frac{b_k}{a_k}} \right)\label{eq:point}
\end{align}
is $\ell_h^{\frac{\alpha}{a_h}}(p^\alpha)^{\frac{b_h}{a_h}}$ if $h\notin J$ and $\left(\frac{\ell_h}{p^{|b_h|}}\right)^{\frac{\alpha}{a_h}}$ if $h\in J$. Also the point in \eqref{eq:point} belongs to the set $\N_{\frac{\alpha}{a_1}}\times\cdots\times \N_{\frac{\alpha}{a_k}}$ because all of the coordinates $h$ for which $h\notin J$ are products of integers, and the coordinates $h$ for which $h\in J$ are integer powers of integers, by assumption. Thus, for $p^{\alpha}>1$, $\bf{n}$ is not $\bf{b}$-visible.

The other direction is analogous to the argument presented in Theorem \ref{thm:ifandonlyif3}, restricted to the entries $h$ for which $h\in J$.  
\end{proof}

\begin{remark}\label{remark:3}
Observe that the last theorem bears a close resemblance to Theorem \ref{thm:ifandonlyif1}, and it as if we had restricted the visibility of $\bf{n}$ to the visibility of those entries of $\bf{n}$ which are in $J$.
\end{remark}

\section{Proportion of $\bb$-visible lattice points}\label{sec:mainresult}

Given the definition and results in the previous section, we  now present our main results. The arguments we make in the proof of Theorem 1 are a generalization of the arguments made by Pinsky in \cite{Pinsky}.
\begin{duplicate}
Let ${\bf{b}}\in \mathbb{N}^k$, such that ${\bf{b}}=(b_1,\dots,b_k)$ satisfies $\gcd(\bb)=1$. Then the proportion of points in $ \mathbb{N}^k$ that are $\bf{b}$-visible is $$\displaystyle\dfrac{1}{\zeta\left(\displaystyle\sum_{i=1}^{k}b_i\right)}.$$
\end{duplicate}
\begin{proof} Let $N\in\mathbb N$ and set $[N]:=\{1,\ldots,N\}$. We want to compute the proportion of points $\mathbf{n}=(n_1,\dots,n_k)\in \mathbb{N}^k$ which are $\bf{b}$-visible. To this end, for each $i=1,\ldots k$, pick $n_i$ integers from the set $[N]$ with uniform probability. This results in a $k$-tuple ${\bf{n}}\in [N]^k$. The associated probability space is $([N]^k, P_N)$, where $P_N$ is the uniform measure.

Fix $\bb=(b_1,b_2,\ldots,b_k)\in \N^k$ and let $p$ be prime. If $C_{p,N}$ denotes the event that the prime $p$ satisfies $p^{b_i}|n_i$ for all $i=1,\ldots,k$, then its probability is given by
\begin{align}
P_N(C_{p,N})&=\frac{1}{N^k}\prod_{i=1}^{k} \left \lfloor{ \frac{N}{p^{b_i}}} \right \rfloor\label{eq:probabilityofevent}
\end{align}
for $p<N$ since there are $\left \lfloor{\frac{N}{p^{b_i}}} \right \rfloor$ numbers divisible by $p^{b_i}$ in $[N]$. 
We now establish that the events $C_{p,N}$ are independent.\\

\noindent{\bf{Claim}.}
Let $R=\prod_{i=1}^rp_i^{b_i}$, where $p_i$ corresponds to the $i$-th prime ($p_1=2$, $p_2=3$,\ldots). If $R|N$, then the events $C_{p_i,N}$ are independent for $i=1,\ldots,r$.\\

\noindent{\emph{Proof of claim.}} If $p_i$ for $i=1,\ldots,r$ are distinct primes, and $m\in[N]$ is arbitrary, then $m$ can only be divisible by $p_i^{b_i}$ for all $i=1\ldots,r$ if $m=\left(\prod_{j=1}^{r}p_j^{b_j}\right)l$,for some $l\in \mathbb{N}$. Note that there must be $\left \lfloor{ \frac{N}{\prod_{j=1}^{r}p_j^{b_j}}} \right \rfloor$ such numbers $m\in[N]$, and the events that $\prod_{j=1}^{r}p_j^{b_j}|n_j$, and $\prod_{j=1}^{r}p_j^{b_j}|n_k$ are independent for all $j\neq k$. Hence

\begin{align}
    P_N\left(\bigcap\limits_{i=1}^{r}C_{p_i,N}\right)&=\frac{1}{N^k}\prod\limits_{i=1}^{k}\left \lfloor{ \frac{N}{\prod_{j=1}^{r}p_j^{b_j}}}\right \rfloor
\end{align}

Moreover, by hypothesis, $R|N$, thus $\left \lfloor{ \frac{N}{\prod_{j=1}^{r}p_j^{b_j}}} \right \rfloor=\frac{N}{\prod_{j=1}^{r}p_j^{b_j}}$. 
The same argument using $p_1,\ldots,p_k$, establishes that there are $\frac{N}{\prod_{j=1}^{r}p_j^{b_j}}$ numbers in $[N]$ divisible by $p_1^{b_1},\ldots,p_k^{b_k}$. Hence

\begin{align}
P_N\left(\bigcap\limits_{i=1}^{r}C_{p_i,N}\right)&=\frac{1}{N^k}\prod\limits_{i=1}^{k}N\left(\prod_{j=1}^{r}\frac{1}{p_j^{b_j}}\right)\\
&=
\prod\limits_{j=1}^{r}\prod_{i=1}^{k}\frac{1}{p_j^{b_j}}\\
&=
\prod\limits_{j=1}^{r}P_N(C_{p_j,N})
\end{align}
as we wanted to show.

We now return to the main proof.
Clearly, for an arbitrary $r\in \mathbb{N}$ and a fixed prime $p_i$
\begin{align}
\bigcap\limits_{i=1}^{\infty}{(C_{p_i,N})}^{c}&=\left(
\bigcap\limits_{i=1}^{r}{(C_{p_i,N})}^{c}\right) \setminus \left(\bigcup\limits_{i=r}^{\infty}C_{p_i,N}\right).\label{eq:relation}
\end{align}

Since $P_N$ is the uniform measure, it always takes finite values. Hence, by the relation in \eqref{eq:relation}, and using the subadditivity property we obtain
\begin{align}
    P_N\left(\bigcap_{i=1}^{r}{(C_{p_i,N})}^{c}\right)&=P_N\left(\bigcap_{i=1}^{\infty}{(C_{p_i,N})}^c\right)+P_N\left(\left(\bigcap_{i=1}^{\infty}{(C_{p_i,N})}^c\right)^{c}\cap \left(\bigcap_{i=1}^{r}{(C_{p_i,N})}^{c}\right)\right)\label{eq:8}
\end{align}
rewriting \eqref{eq:8} yields
\begin{align}
    P_N\left(\bigcap_{i=1}^{\infty}{(C_{p_i,N})}^c\right)
    &=
    P_N\left(\bigcap_{i=1}^{r}{(C_{p_i,N})}^{c}\right)-P_N\left((\bigcap_{i=1}^{\infty}{(C_{p_i,N})}^{c})^{c}\cap (\bigcap_{i=1}^{r}{(C_{p_i,N})}^{c})\right)\\
    &\geq
    P_N\left(\bigcap_{i=1}^{r}{(C_{p_i,N})}^{c}\right)-P_N\left(\bigcup_{i=1}^{\infty}C_{p_i,N}\right)
\end{align}
and lastly
\begin{align}
    P_N\left(\bigcap\limits_{i=1}^{r}{(C_{p_i,N})}^{c}\right)-P_N\left(\bigcup\limits_{i=r}^{\infty}C_{p_i,N}\right)
    &\leq 
    P_N\left(\bigcap\limits_{i=1}^{\infty}{(C_{p_i,N})}^{c}\right)\leq \sum\limits_{i=1}^{\infty}P_N(C_{p_i,N}).\label{eq:11}
\end{align}

Furthermore, by \eqref{eq:probabilityofevent} we have that 
\begin{align}
    \sum\limits_{i=1}^{\infty}P_N(C_{p_i,N})
&=
\sum\limits_{i=1}^{\infty}\left(\frac{1}{N^k} \prod_{j=1}^{k}\left \lfloor{\frac{N}{p_j^{b_j}}} \right \rfloor\right) \leq
\sum\limits_{i=1}^{\infty}\frac{1}{\prod_{j=1}^{k}p_j^{b_j}}.
\end{align}

If we let $R=\prod_{i=1}^rp_i^{b_i}$, and $N$ is such that $R|N$, we have by the previous claim that the events $C_{p_i,N}$ are independent for $i=1,\ldots,r$. This implies that the events ${(C_{p_i,N})}^{c}$ are also independent. Then, $P_N\left(\bigcap_{i=1}^{r}{(C_{p_i,N})}^{c}\right)=\prod_{i=1}^{r}\left(1-\frac{1}{\prod_{j=1}^{k}p_i^{b_j}}\right)$. 

By \eqref{eq:11}, if $R|N$ we have that
\begin{align}\prod_{i=1}^{r}\left(1-\frac{1}{\prod_{j=1}^{k}p_i^{b_j}}\right)- \sum\limits_{i=r}^{\infty}\frac{1}{\prod_{j=1}^{k}p_j^{b_j}}    \leq P_N\left(\bigcap\limits_{i=1}^{\infty}{(C_{p_i,N})}^{c}\right)\leq 
\prod_{i=1}^{r}\left(1-\frac{1}{\prod_{j=1}^{k}p_i^{b_j}}\right).\label{eq:line}
\end{align}
Recall that $P_N\left(\bigcap\limits_{i=1}^{\infty}{(C_{p_i,N})}^{c}\right)$ is the probability that the $k$-tuple $\nn$ is $\bb$-visible. Thus, $$N^kP_N\left(\bigcap\limits_{i=1}^{\infty}{(C_{p_i,N})}^{c}\right),$$ is the number of lattice points $\bf{n}$ that are ${\bf{b}}$-visible and belong to $[N]^k$.

Now let $N$ be arbitrary ($R\nmid N$ may happen). There exist $k_1,k_2\in \mathbb{N}$ satisfying the inequality $Rk_1< N < Rk_2$, and $k_1\geq k_1'$, $k_2\leq k_2'$ for any other pair $k_1',k_2'\in \mathbb{N}$ that satisfies the inequality.
For that pair $k_1,k_2$ we have by election that $Rk_2-N<R$ and $N-Rk_1<R$. Hence,
 $N-R<Rk_1$
and
$Rk_2<N+R$.

Moreover, because there are more ${\bf{b}}$-visible points in $[N+1]^k$ than in $[N]^k$, the expression $N^kP_N\left(\bigcap\limits_{i=1}^{\infty}{(C_{p_i,N})}^{c}\right)$ is increasing in $N$. In particular if ${N_1}:=Rk_1$, ${N_2}:=Rk_2$, then we have:

\begin{align*}
    {{N_1}}^kP_{N_1}\left(\bigcap\limits_{i=1}^{\infty}{(C_{p_i,N_1})}^{c}\right)
    &\leq 
{N}^kP_N\left(\bigcap\limits_{i=1}^{\infty}{(C_{p_i,N})}^{c}\right)\leq
{{N_2}}^kP_{{N_2}}\left(\bigcap\limits_{i=1}^{\infty}{(C_{p_i,N_2})}^{c}\right)
\end{align*}
and hence
\begin{align}
    \frac{{{N_1}}^k}{N^k}P_{N_1}\left(\bigcap\limits_{i=1}^{\infty}{(C_{p_i,N_1})}^{c}\right)
    &\leq 
P_N\left(\bigcap\limits_{i=1}^{\infty}{(C_{p_i,N})}^{c}\right)\leq
\frac{{{N_2}}^k}{N^k}P_{N_2}\left(\bigcap\limits_{i=1}^{\infty}{(C_{p_i,N_2})}^{c}\right).\label{eq:14}
\end{align}

Since ${N_1}=Rk_1$ and ${N_2}=Rk_2$, then ${N_1}\leq N \leq {N_2}$ holds because of the choice of $k_1$ and $k_2$. It is immediate from the definition of ${N_1}$ and ${N_2}$ that $R|{N_1}$ and $R|{N_2}$. As a consequence of \eqref{eq:line} it follows that \eqref{eq:14} becomes
\begin{align}
\frac{{{N_1}}^k}{N^k}\left(\prod_{i=1}^{r}\left(1-\frac{1}{\prod_{j=1}^{k}p_i^{b_j}}\right)- {\sum\limits_{i=r}^{\infty}\frac{1}{\prod_{j=1}^k p_j^{b_j}}}\right)\leq 
P_N\left(\bigcap\limits_{i=1}^{\infty}{(C_{p_i,N})}^{c}\right)\leq
\frac{{{N_2}}^k}{N^k}\prod_{i=1}^{r}\left(1-\frac{1}{\prod_{j=1}^{k}p_i^{b_j}}\right).\label{eq:last}
\end{align}

Using  Theorem \ref{thm:ifandonlyif1}, finding the proportion of $\bf{b}$-visible points $\N^k$ is equivalent to finding $$P_N\left(\bigcap\limits_{i=1}^{\infty}\left(C_{p_i,N}\right)^{c}\right)$$ as $N$ goes to infinity, which is equivalent to letting $N\to \infty$ in the last inequality of \eqref{eq:last}, and then letting $r\to \infty$. Doing so yields

\begin{align*}
    \lim_{N\to \infty}P_N\left(\bigcap\limits_{i=1}^{\infty}{(C_{p_i,N})}^{c}\right)&=\prod_{i=1}^{\infty}\left(1-\frac{1}{\prod_{j=1}^{k}p_i^{b_j}}\right)
    =
    {\prod_{i=1 }^{\infty}\left(1-\frac{1}{{p_i}^{\sum_{i=1}^{k}b_i}}\right)}
    =\frac{1}{\zeta\left(\sum_{i=1}^{k}b_i\right)}.\qedhere
\end{align*}
\end{proof}

We conclude by noticing that Theorem \ref{thm:main} recovers the results obtained in $b$-visibility, by taking ${\bf{b}}=(1,b)$, and the generalization of the classical lattice point visibility in $n$-dimensions by taking ${\bf{b}}={\bf{1}}=(1,1,\ldots,1)\in\N^n$.

We now generalize the results obtained by Harris and Omar \cite{HarrisOmar2017} regarding $b$-visibility allowing rational exponents. 
\begin{duplicate}
Fix ${\bf{b}}=(\frac{b_1}{a_1},\frac{b_2}{a_2},\ldots,\frac{b_k}{a_k})\in{\mathbb{Q}_{>0}}^k$, with $\bf{b}$ satisfying $\gcd(\bb)=1$. Then the proportion of points in $\mathbb{N}_{\frac{\alpha}{a_1}}\times \ldots \times \mathbb{N}_{\frac{\alpha}{a_n}}$ that are $\bf{b}$-visible is $\frac{1}{\zeta(\sum_{i=1}^kb_i)}$.
\end{duplicate}
\begin{proof}
Define $[N_{\frac{\alpha}{a_i}}]:=\{1^{\frac{\alpha}{a_i}},2^{\frac{\alpha}{a_i}},\ldots, {\left \lfloor{N^{\frac{a_i}{\alpha}}} \right \rfloor}^{\frac{\alpha}{a_i}} \}$ for $i=1,\ldots,k$. Let $n_1^{\frac{\alpha}{a_1}},\ldots,n_k^{\frac{\alpha}{a_k}}$ be picked independently with uniform probability in $[N_{\frac{\alpha}{a_1}}],\ldots$, $[N_{\frac{\alpha}{a_k}}]$, respectively.

Fix a prime $p$. Then the probability that $p^{b_i}|n_i^{\frac{\alpha}{a_i}}$, is thus given by \[\frac{1}{\left \lfloor{N^{\frac{a_i}{\alpha}}}\right \rfloor}\cdot \left \lfloor{\frac{\left \lfloor{N^{\frac{a_i}{\alpha}}}\right \rfloor}{p^{b_i}}} \right \rfloor.\]
By mutual independence, the probability $P_{p,N}$ that $p^{b_i}|n_i^{\frac{\alpha}{a_i}}$ for all $i=1,\ldots,k$ is given by
$$\prod\limits_{i=1}^k\frac{1}{\left \lfloor{N^{\frac{a_i}{\alpha}}}\right \rfloor}\cdot \left \lfloor{\frac{\left \lfloor{N^{\frac{a_i}{\alpha}}}\right \rfloor}{p^{b_i}}} \right \rfloor.$$
And since $P_{p,N}\to \prod\limits_{i=1}^k\frac{1}{p^{b_i}}=\frac{1}{p^{\sum_{i}b_i}}$ as $N\to \infty$, it is a consequence of Theorem \ref{thm:ifandonlyif3} that
\[\lim_{N\to \infty}\prod_{\substack{\text{$p$ prime}\\p\leq N}}(1-P_{p,N})=\prod_{\text{$p$ prime}}(1-\frac{1}{p^{\sum_{i}b_i}})=\frac{1}{\zeta(\sum_{i}b_i)}.\qedhere\]
\end{proof}

We remark that in \cite{HarrisOmar2017}, Harris and Omar considered the case when $\bb=(1,b/a)$ with $b/a\in\Q$. When $b$ was a negative integer, they thought of a point $(r,s)\in\N^2$ as being visible (or invisible) from a point at infinity, i.e. located at $(\infty,0)$. From this they established that the proportion of visible lattice points in $\N_a\times \N$, where $\N_a=\{1^a,2^a,3^a,\ldots\}$, was given by $1/{\zeta(b)}$. However, their study considered the number-theoretic approach where the point $(r,s)$ was $(b/a)$-visible provided there was no prime $p$ such that $p^b|s$. Taking this number-theoretic approach we now establish our final result.

\begin{duplicate}
Let ${\bf{b}}=(\frac{b_1}{a_1},\ldots,\frac{b_k}{a_k})\in({\mathbb{Q}^{*}})^k$ be such that its negative entries are indexed by the set $J \subseteq [k]$, with $\bf{b}$ satisfying the conditions that $a_1,a_2,\ldots,a_k\in\mathbb{N}$ and  $\gcd(\frac{b_1}{a_1},\frac{b_2}{a_2},\ldots,\frac{b_k}{a_k})=1$. Then the proportion of points in $\mathbb{N}_{\frac{\alpha}{a_1}}\times \cdots \times \mathbb{N}_{\frac{\alpha}{a_n}}$ that are $\bf{b}$-visible is $\frac{1}{\zeta(\sum_{j\in J}|b_j|)}$.
\end{duplicate}
\begin{proof}
The proof of this theorem is analogous to the proof of the previous theorem, and is a consequence of Theorem \ref{thm:ifandonlyif4} and Remark \ref{remark:3}.
\end{proof}

\begin{remark}
Recall that in \cite{HarrisOmar2017}, Harris and Omar obtained that the density of $\bf{b}$-visible points for ${\bf{b}}=\frac{b}{a}$, was $\frac{1}{\zeta(b)}$ if $b$ is negative, whereas the density is $\frac{1}{\zeta(b+1)}$ if $b$ is positive. Our last theorem explains the reason for the difference between these densities: when the vector $\bf{b}$ contains both positive and negative entries, only the negative entries matter. 
\end{remark}


\begin{thebibliography}{99}


\bibitem{apostol}T.~M. Apostol, \emph{Introduction to Analytic Number Theory}, Springer, New York (1976).

\bibitem{Chen}
Y.~Chen, L.~Cheng, Visibility of lattice points, \emph{Acta Arith.}  \textbf{107} (2003) 203--207,
\underline{\url{http://dx.doi.org/10.4064/aa107-3-1}}.

\bibitem{Christopher} 
J.~Christopher, \emph{The Asymptotic Density of Some $k$-Dimensional Sets}
The American Mathematical Monthly, Vol. 63, No. 6 pp. 399-401.


\bibitem{GHKM2017}
E.~Goins, P.~E. Harris, B. Kubik, A.~Mbirika, \emph{Lattice point visibility on generalized lines of sight}. The American Mathematical Monthly (2018), 125:7, 593-601.
\bibitem{Mbirika2015}
A.~Goodrich, A.~Mbirika, J.~Nielsen, \emph{New methods to find patches of invisible integer lattice points},  Preprint available at \underline{\url{https://people.uwec.edu/mbirika/lattice_point_paper.pdf}}.

\bibitem{HarrisOmar2017}
P.~E. Harris, M.~Omar, \emph{Lattice point visibility on power functions}. To appear in Integers.
Preprint available at \underline{\url{https://arxiv.org/pdf/1712.09155.pdf}}.

\bibitem{Laishram}
S.~Laishram, F.~Luca, Rectangles of nonvisible lattice points, \emph{J. Integer Seq.} \textbf{18} (2015)
Article 15.10.8, 11. 

\bibitem{Laison}
J.~D.~Laison, M.~Schick, Seeing dots: visibility of lattice points, \emph{Math. Mag.} \textbf{80} (2007) 274--282.

\bibitem{Nicholson}
N.~Nicholson, R.~Rachan, On weak lattice point visibility, \emph{Involve} \textbf{9} (2016) 411--414,
\underline{\url{http://dx.doi.org/10.2140/involve.2016.9.411}}.

\bibitem{Pinsky}
R.~G. Pinsky, \emph{Problems from the Discrete to the Continuous}, Springer, Cham (2014).


\bibitem{Schumer}
P.~Schumer,
Strings of strongly composite integers and invisible lattice points,
\emph{College Math. J.} \textbf{21} (1990) 37--40,
\underline{\url{http://dx.doi.org/10.2307/2686720}}.
\end{thebibliography}
\end{document}